\documentclass[11pt]{amsart}

\usepackage{amsmath,amssymb,amsthm,mathtools}
\usepackage{enumitem}
\usepackage{hyperref}

\newcommand{\PP}{\mathbb P}
\newcommand{\CC}{\mathbb C}
\newcommand{\RR}{\mathbb R}
\newcommand{\ZZ}{\mathbb Z}
\newcommand{\cO}{\mathcal O}
\newcommand{\F}{\mathbb F}
\newcommand{\eps}{\varepsilon}
\newcommand{\Bl}{\operatorname{Bl}}
\newcommand{\NS}{\operatorname{NS}}
\newcommand{\Pic}{\operatorname{Pic}}
\newcommand{\Span}{\operatorname{Span}}

\newtheorem{theorem}{Theorem}[section]
\newtheorem{proposition}[theorem]{Proposition}
\newtheorem{lemma}[theorem]{Lemma}
\newtheorem{corollary}[theorem]{Corollary}
\newtheorem{remark}[theorem]{Remark}
\newtheorem{question}[theorem]{Question}

\title[Dihedral reflections and irrational Seshadri constants]{Dihedral reflections and an infinite series of irrational Seshadri constants}
\author[G. Malara at al.]{Grzegorz Malara, Łukasz Merta, Justyna Szpond and Marcin Zieli\'nski}
\begin{document}

\begin{abstract}
Laface and Ugaglia recently constructed an irrational one-point Seshadri constant on the blow-up of $\PP^2$ at nine very general points by combining a dihedral orbit on $\PP^1\times\PP^1$, a sequence of de Jonqui\`eres transformations, and a reflection argument along a $(-4)$-curve with balanced normal bundle.
We show that the same mechanism extends uniformly to every odd integer $n\geq 5$.
For $n=2k+1$ we prove
$$
\eps\bigl(\cO_{\PP^1\times\PP^1}(n-4,1);p_1,\ldots,p_{2n}\bigr)=\sqrt{\frac{n-4}{n}}
$$
at $2n$ very general points, and already at a very general free orbit of a fixed dihedral group of order $2n$.
For every odd $n\geq 7$ this produces an explicit ample line bundle on the blow-up of $\PP^2$ at $k+7=(n+13)/2$ very general points whose one-point Seshadri constant at a very general point equals
$$
2\sqrt{n(n-4)}.
$$
More precisely, after one quadratic transformation we obtain the ample divisor
$$
L_n=(3n-4)H-nE_1-(n-2)(E_2+\cdots+E_5)-4(E_6+\cdots+E_{k+5})-2(E_{k+6}+E_{k+7}),
$$
with $L_n^2=4n(n-4)$ and $\eps(L_n;x)=\sqrt{L_n^2}$.
We also isolate an abstract balanced-reflection principle underlying the construction: a nef class on a special fiber can be reflected across a rational curve of square $-2a$ whenever the curve has normal bundle $\cO_{\PP^1}(-a)^{\oplus 2}$ in the total space, and the reflected class is nef on very general fibers.
\end{abstract}

\maketitle

\section{Introduction}

Let $L$ be an ample line bundle on a smooth projective surface $X$ and let $x\in X$.
The one-point Seshadri constant is
$$
\eps(L;x)=\sup\{t\geq 0:\pi^*L-tE_x\text{ is nef}\},
$$
where $\pi:\Bl_xX\to X$ is the blow-up and $E_x$ is its exceptional curve.
One always has the volume bound
$$
\eps(L;x)\leq \sqrt{L^2}.
$$
If the inequality is strict, then on a surface the Seshadri constant is computed by a curve and is therefore rational.
Consequently an irrational one-point Seshadri constant must attain the volume bound and $L^2$ must be a nonsquare.

Laface and Ugaglia recently produced the first unconditional example of this phenomenon on a blow-up of the plane at very general points \cite{LU}.
Their construction proves
$$
\eps\bigl(\cO_{\PP^1\times\PP^1}(1,1);p_1,\ldots,p_{10}\bigr)=\frac1{\sqrt5}
$$
for ten very general points and then converts this multipoint equality into the one-point equality
$$
\eps\bigl(9H-3(E_1+\cdots+E_5)-2(E_6+\cdots+E_9);x\bigr)=2\sqrt5
$$
on the blow-up of $\PP^2$ at nine very general points.
Their proof completes, in the ten-point case, the reflection strategy developed by Dionne and Roth for Seshadri constants on $\PP^1\times\PP^1$ \cite{DR}.

The purpose of the present note is to observe that the dihedral construction is not restricted to ten points.
The special curve, the de Jonqui\`eres dynamics, the limiting irrational nef ray, and the balanced normal-bundle calculation all admit a uniform form for $2n$ points when $n$ is odd.
The resulting statement is the following.

\begin{theorem}\label{thm:F0-main}
Let $n\geq 5$ be odd and let $p_1,\ldots,p_{2n}$ be very general points of $\F_0=\PP^1\times\PP^1$.
Then
$$
\eps\bigl(\cO_{\F_0}(n-4,1);p_1,\ldots,p_{2n}\bigr)=\sqrt{\frac{n-4}{n}}.
$$
The same equality holds when $p_1,\ldots,p_{2n}$ form a very general free orbit of the fixed dihedral group $G_n$ defined in Section~\ref{sec:dihedral}.
\end{theorem}

The case $n=5$ is precisely the multipoint theorem of Laface--Ugaglia.
For $n\geq 7$ the theorem gives new irrational constants already on $\PP^1\times\PP^1$.
More importantly, the orbit formulation descends to a one-point problem on the quotient and then lifts, after resolving the quotient and deforming the plane model, to a smooth rational surface.

Write $n=2k+1$.
Our main smooth-surface statement is the following.

\begin{theorem}\label{thm:plane-main}
Let $n=2k+1\geq 7$, let
$$
r=k+7=\frac{n+13}{2},
$$
and let $X_r=\Bl_{p_1,\ldots,p_r}\PP^2$ for very general points $p_i$.
Then the divisor
\begin{equation}\label{eq:Ln}
L_n=(3n-4)H-nE_1-(n-2)(E_2+\cdots+E_5)-4(E_6+\cdots+E_{k+5})-2(E_{k+6}+E_{k+7})
\end{equation}
is ample and satisfies
$$
\eps(L_n;x)=2\sqrt{n(n-4)}=\sqrt{L_n^2}
$$
for a very general point $x\in X_r$.
In particular, $\eps(L_n;x)$ is irrational.
\end{theorem}

The first new case is $n=7$, where $r=10$ and
$$
L_7=17H-7E_1-5(E_2+E_3+E_4+E_5)-4(E_6+E_7+E_8)-2(E_9+E_{10}),
$$
so that
$$
L_7^2=84,\\;\; \eps(L_7;x)=2\sqrt{21}.
$$

There are two conceptual points in the proof.
The first is a dynamical construction of an irrational nef isotropic class by iterating explicit de Jonqui\`eres transformations.
The second is a deformation-theoretic reflection across a rational $(-4)$-curve whose normal bundle in the threefold total space is $\cO_{\PP^1}(-2)^{\oplus2}$.
The second point is sufficiently general to be isolated independently of the dihedral situation.
We do this in Section~\ref{sec:principle}, where we formulate what may be viewed as the Laface--Ugaglia principle in abstract form.

The proof of Theorem~\ref{thm:F0-main} occupies Sections~\ref{sec:dihedral}--\ref{sec:reflection}.
In Sections~\ref{sec:quotient}--\ref{sec:plane} we analyze the quotient for arbitrary odd $n$, resolve its six singularities, construct a uniform plane model, and deform the corresponding nef class to very general plane points.
Finally, Section~\ref{sec:ampleness} proves ampleness of \eqref{eq:Ln} using the criterion of Hanumanthu \cite{Hanumanthu} and completes the proof of Theorem~\ref{thm:plane-main}.

\section{The balanced-reflection principle}\label{sec:principle}

We begin with the deformation-theoretic mechanism that turns a nef class on a special fiber into a different nef class on very general fibers.
The statement is an abstract version of the multiplicity-reflection argument used in \cite{LU}, ultimately based on Petrakiev's specialization method \cite{Petrakiev}.

\begin{proposition}[Balanced reflection principle]\label{prop:balanced-reflection}
Let $B$ be a smooth connected complex curve with distinguished point $0$, and let $f:\mathcal X\to B$ be a smooth projective threefold whose fibers are smooth projective surfaces.
Assume that the N\'eron--Severi spaces of the fibers are identified by a marking.
Suppose that the special fiber $X_0$ contains a smooth rational curve $\Gamma$ with class $\gamma$ such that
$$
\gamma^2=-2a,\\;\; N_{\Gamma/\mathcal X}\simeq \cO_{\PP^1}(-a)\oplus\cO_{\PP^1}(-a)
$$
for some integer $a>0$.
If $D\in\NS(X_0)_{\RR}$ is nef, then the marked class
\begin{equation}\label{eq:reflection}
s_\gamma(D)=D+\frac{D\cdot\gamma}{a}\,\gamma
\end{equation}
is nef on very general fibers of $f$.
\end{proposition}

\begin{proof}
Let $c$ be an effective integral class on the geometric generic fiber.
After a finite extension of the base field and specialization, $c$ is represented on $X_0$ by an effective divisor $C_0$.
Let $\mu$ be the order of vanishing of the corresponding relative divisor along $\Gamma$.
The first nonzero normal term gives a nonzero homomorphism
$$
\cO_{\PP^1}(-c\cdot\gamma)\longrightarrow \operatorname{Sym}^{\mu}N_{\Gamma/\mathcal X}^{\vee}\simeq \cO_{\PP^1}(a\mu)^{\oplus(\mu+1)}.
$$
Hence
$$
-c\cdot\gamma\leq a\mu,
$$
while $C_0-\mu\Gamma$ is effective.
Therefore
$$
s_\gamma(c)=c+\frac{c\cdot\gamma}{a}\gamma=(c-\mu\gamma)+\left(\mu+\frac{c\cdot\gamma}{a}\right)\gamma
$$
is represented by an effective real divisor on $X_0$.
The reflection $s_\gamma$ is self-adjoint for the intersection form, so
$$
s_\gamma(D)\cdot c=D\cdot s_\gamma(c)\geq0.
$$
Thus $s_\gamma(D)$ is nef on the geometric generic fiber.
Nefness of a fixed marked class then holds on very general fibers, since the relative Hilbert scheme has only countably many components and the locus where a class of negative intersection becomes effective is a proper closed subset.
\end{proof}

\begin{remark}\label{rem:reflection-isometry}
The map $s_\gamma$ is the orthogonal reflection in the hyperplane $\gamma^\perp$.
Indeed, since $\gamma^2=-2a$, formula \eqref{eq:reflection} is the usual expression
$$
s_\gamma(D)=D-2\frac{D\cdot\gamma}{\gamma^2}\gamma.
$$
In particular, it preserves the intersection form and sends isotropic classes to isotropic classes.
\end{remark}

The following formulation explains how Proposition~\ref{prop:balanced-reflection} can produce irrational Seshadri constants.

\begin{corollary}[Laface--Ugaglia principle]\label{cor:LU-principle}
In the situation of Proposition~\ref{prop:balanced-reflection}, suppose that after identifying a very general fiber with the blow-up $\pi:\widetilde X\to X$ of a smooth projective surface at a point, the reflected class is a positive multiple of
$$
\pi^*L-\sqrt{L^2}\,E
$$
for an ample line bundle $L$ on $X$.
Then
$$
\eps(L;x)=\sqrt{L^2}
$$
for a very general point $x\in X$.
If $L^2$ is not a square, this Seshadri constant is irrational.
\end{corollary}

\begin{proof}
Proposition~\ref{prop:balanced-reflection} gives nefness of the displayed class on a very general fiber, hence $\eps(L;x)\geq\sqrt{L^2}$.
The reverse inequality is the volume bound.
\end{proof}

The significance of the principle is that the difficult nefness problem may be transferred to a special fiber where a dynamical construction is available.
The normal-bundle condition then converts that special nef ray into the ray relevant to the desired Seshadri constant.

\section{Dihedral orbits and de Jonqui\`eres transformations}\label{sec:dihedral}

Fix an odd integer $n\geq5$, a primitive $n$-th root of unity $\zeta$, and a real number $u>1$.
On $S=\F_0=\PP^1\times\PP^1$, with affine coordinates $(x,z)$, consider
\begin{equation}\label{eq:group-original}
r(x,z)=(\zeta x,\zeta^{n-2}z),\\;\; s(x,z)=\left(\frac ux,\frac{u^{n-2}}z\right).
\end{equation}
These automorphisms satisfy
$$
r^n=s^2=1,\\;\; srs=r^{-1},
$$
so they generate a dihedral group $G_n$ of order $2n$.
The orbit $G_n\cdot(1,1)$ is free and consists of the points
$$
(\zeta^\ell,\zeta^{(n-2)\ell}),\\;\; (u\zeta^\ell,u^{n-2}\zeta^{(n-2)\ell}),\\;\; 0\leq\ell<n.
$$
All these points lie on the smooth rational curve
\begin{equation}\label{eq:Gamma}
\Gamma=\{z=x^{n-2}\},
\end{equation}
which has class $(n-2)H_1+H_2$.

We now construct a sequence of isomorphisms between blow-ups of $S$ at suitable unions of two cyclic orbits.
For a unit $w\in(\ZZ/n\ZZ)^*$ define
$$
r_w(x,z)=(\zeta x,\zeta^w z)
$$
and, for $c,d\in\CC^*$, put
$$
Z_w(c,d)=\langle r_w\rangle\cdot(1,1)\cup\langle r_w\rangle\cdot(c,d).
$$

\begin{lemma}[Uniform de Jonqui\`eres transformation]\label{lem:dJ}
Assume that
$$
c^n\neq1,\\;\; d\notin\{c^w,c^{w-n}\}.
$$
Then there is a birational map $\psi:S\dashrightarrow S$ which lifts to an isomorphism
$$
\widetilde\psi:\Bl_{Z_w(c,d)}S\xrightarrow{\sim}\Bl_{Z_w(c,d^*)}S,
\\;\; d^*=\frac{c^{2w-n}}d.
$$
If $H_1,H_2,E_i$ and $H_1',H_2',E_i'$ denote the natural markings on source and target, with corresponding centers labeled by the same first coordinate, then
\begin{equation}\label{eq:dJ-marking}
\widetilde\psi^*H_1'=H_1,\\;\; \widetilde\psi^*H_2'=nH_1+H_2-E,\\;\; \widetilde\psi^*E_i'=H_1-E_i,
\end{equation}
where $E=\sum_{i=1}^{2n}E_i$.
\end{lemma}

\begin{proof}
The pullback action of $r_w$ on $H^0(S,\cO_S(n,1))$ has two three-dimensional weight spaces,
$$
\langle1,x^n,x^{n-w}z\rangle,\\;\; \langle x^w,z,x^nz\rangle,
$$
and all remaining weight spaces have dimension two.
Vanishing at $(1,1)$ and $(c,d)$ determines, up to scalar, the sections
\begin{align*}
\sigma_0&=c^n-dc^{n-w}+(dc^{n-w}-1)x^n+(1-c^n)x^{n-w}z,\\
\sigma_1&=d(c^n-1)x^w+(c^w-dc^n)z+(d-c^w)x^nz.
\end{align*}
After bihomogenization, write $\sigma_i=A_i z_0+B_i z_1$.
A direct determinant calculation gives
\begin{equation}\label{eq:determinant}
A_0B_1-A_1B_0=-(c^w-d)(c^{n-w}d-1)(x_1^n-x_0^n)(x_1^n-c^nx_0^n).
\end{equation}
Under our assumptions the right-hand side has precisely $2n$ simple zeros in the first factor, and the coefficient matrix has rank one at each zero.
Since $(nH_1+H_2)^2=2n$, the base scheme of the pencil $\langle\sigma_0,\sigma_1\rangle$ is the reduced set $Z_w(c,d)$.
Thus
$$
\psi(x,z)=\left(x,\frac{\sigma_1(x,z)}{d\sigma_0(x,z)}\right)
$$
is resolved by blowing up these $2n$ points and contracts the strict transforms of the $2n$ first-ruling fibers through them.
The target is again $S$, blown up at the images of the contracted fibers.
Evaluating on $x=1$ and $x=c$ gives second coordinates $1$ and $c^{2w-n}/d$, respectively, hence the target centers form $Z_w(c,d^*)$.
The formulas \eqref{eq:dJ-marking} follow from the pencil and from the classes of the contracted curves.
\end{proof}

After applying Lemma~\ref{lem:dJ}, exchange the two factors of the target.
The parameters transform by
\begin{equation}\label{eq:param-transform}
(c,d,w)\longmapsto\left(c^{2w-n}/d,c,w^{-1}\right),
\end{equation}
where the inverse is taken modulo $n$.
Set
$$
w_0=\frac{n-1}{2}.
$$
Since $w_0^{-1}\equiv n-2\pmod n$, the weights alternate between $w_0$ and $n-2$.
Start with
\begin{equation}\label{eq:initial-triple}
(c_0,d_0,w_0)=\left(u^{n-2},u,\frac{n-1}{2}\right).
\end{equation}
Write
$$
c_j=u^{\rho_j},\\;\; d_j=u^{\tau_j}.
$$
At the two alternating steps the exponent pairs transform as
\begin{equation}\label{eq:exponent-recurrence}
(\rho,\tau)\longmapsto(-\rho-\tau,\rho),\\;\; (\rho,\tau)\longmapsto((n-4)\rho-\tau,\rho).
\end{equation}

\begin{lemma}\label{lem:nondegenerate}
For every $j\geq0$ the hypotheses of Lemma~\ref{lem:dJ} are satisfied by $(c_j,d_j,w_j)$.
\end{lemma}

\begin{proof}
Starting from $(\rho_0,\tau_0)=(n-2,1)$, induction using \eqref{eq:exponent-recurrence} shows that
$$
|\rho_j|>|\tau_j|>0
$$
for all $j$.
At even steps the two exponents have the same sign, and at odd steps they have opposite signs.
Indeed, the first transformation sends a same-sign pair with $|\rho|>|\tau|$ to an opposite-sign pair with the same strict inequality, while the second transformation sends an opposite-sign pair to a same-sign pair because $n-4\geq1$, and for $n\geq5$ the first component has strictly larger absolute value than the second.
Thus $c_j^n\neq1$.
Moreover the forbidden equalities $\tau_j=w_j\rho_j$ and $\tau_j=(w_j-n)\rho_j$ are impossible because $|\tau_j/\rho_j|<1$, whereas
$$
|w_j|>1,\\;\; |w_j-n|>1
$$
for $w_j\in\{(n-1)/2,n-2\}$.
\end{proof}

Let $P_j$ be the factor exchange of $Z_{w_j}(c_j,d_j)$ and let
$$
X_j=\Bl_{P_j}S.
$$
The preceding construction gives isomorphisms $f_j:X_j\to X_{j+1}$.
The initial configuration is
$$
P_0=G_n\cdot(1,1).
$$
On the symmetric marked subspace $\Span(H_1,H_2,E)$, every pullback $f_j^*$ is represented by
\begin{equation}\label{eq:Tn}
T_n=\begin{pmatrix}
0&1&0\\
1&n&2n\\
0&-1&-1
\end{pmatrix}.
\end{equation}

\section{The limiting nef ray}\label{sec:nef-ray}

Put
$$
\Delta=\sqrt{n(n-4)},\\;\; \alpha=\frac{n-\Delta}{2},\\;\; \beta=\frac{n+\Delta}{2},
$$
and consider the class
\begin{equation}\label{eq:Rn}
R_n=\alpha H_1+\beta H_2-E
\end{equation}
on $X_0$.
Notice that
$$
\alpha\beta=n,
$$
so $R_n^2=0$.

\begin{proposition}\label{prop:R-nef}
The class $R_n$ is nef on $X_0=\Bl_{G_n\cdot(1,1)}S$.
\end{proposition}

\begin{proof}
The characteristic polynomial of $T_n$ is
$$
(t-1)\bigl(t^2-(n-2)t+1\bigr).
$$
Its dominant eigenvalue is
$$
\lambda=\frac{n-2+\Delta}{2}>1.
$$
Set
$$
\overline R_n=\beta H_1+\alpha H_2-E,\\;\; K=-2H_1-2H_2+E.
$$
A direct calculation gives
$$
T_n(R_n)=\lambda R_n,\\;\; T_n(\overline R_n)=\lambda^{-1}\overline R_n,\\;\; T_n(K)=K.
$$
Moreover
\begin{equation}\label{eq:H2-decomposition}
H_2=\frac{\beta}{n(n-4)}R_n+\frac{\alpha}{n(n-4)}\overline R_n+\frac1{n-4}K.
\end{equation}
For every $j$, the class $T_n^j(H_2)$ is the pullback of a ruling class under a composition of the isomorphisms $f_i$, hence is nef.
Multiplying \eqref{eq:H2-decomposition} by $\lambda^{-j}$ after applying $T_n^j$ shows that a positive multiple of $R_n$ is the limit of nef classes.
Since the nef cone is closed, $R_n$ is nef.
\end{proof}

The strict transform of the curve \eqref{eq:Gamma} on $X_0$ has class
\begin{equation}\label{eq:gamma}
\gamma=(n-2)H_1+H_2-E.
\end{equation}
Its self-intersection is
$$
\gamma^2=2(n-2)-2n=-4.
$$
The reflection in this $(-4)$-class is therefore
$$
s_\gamma(D)=D+\frac{D\cdot\gamma}{2}\gamma.
$$
The next identity is the numerical heart of the construction.

\begin{lemma}\label{lem:reflected-ray}
The reflected class $s_\gamma(R_n)$ is a positive multiple of
\begin{equation}\label{eq:Seshadri-ray}
(n-4)H_1+H_2-\sqrt{\frac{n-4}{n}}\,E.
\end{equation}
\end{lemma}

\begin{proof}
Substitute \eqref{eq:Rn} and \eqref{eq:gamma} into the reflection formula.
Using $\Delta^2=n(n-4)$ gives
$$
s_\gamma(R_n)=c_n\left((n-4)H_1+H_2-\frac{\Delta}{n}E\right),
$$
where
$$
c_n=\frac{n(n-3)+(n-1)\Delta}{4}>0.
$$
Since $\Delta/n=\sqrt{(n-4)/n}$, this is \eqref{eq:Seshadri-ray} up to the positive scalar $c_n$.
\end{proof}

\section{The balanced normal bundle and the multipoint theorem}\label{sec:reflection}

Keep the group $G_n$ fixed and move $(1,1)$ transversally by
$$
p_t=(1+t,1).
$$
For $t$ in a small disc $B$ around $0$, the orbit $G_n\cdot p_t$ consists of $2n$ pairwise distinct points.
Blow up the corresponding labeled sections of $S\times B$ and denote the resulting smooth threefold by
$$
\mathcal X\longrightarrow B.
$$
Let $\Gamma_0$ be the strict transform in the special fiber of the curve $\Gamma$ from \eqref{eq:Gamma}.

\begin{lemma}\label{lem:normal-bundle}
One has
$$
N_{\Gamma_0/\mathcal X}\simeq \cO_{\PP^1}(-2)\oplus\cO_{\PP^1}(-2).
$$
\end{lemma}

\begin{proof}
Before blowing up the moving sections, the normal bundle of $\Gamma\simeq\PP^1$ in $S\times B$ is
$$
\cO_{\PP^1}(2n-4)\oplus\cO_{\PP^1},
$$
because $\Gamma^2=2(n-2)$ in $S$ and the parameter direction contributes the trivial summand.
Each of the $2n$ blow-ups performs an elementary transform and lowers the total degree by one.
Hence, for the resulting rank-two bundle $N=N_{\Gamma_0/\mathcal X}$,
$$
\deg N=(2n-4)-2n=-4.
$$
It remains to prove $H^0(N(1))=0$.

Use the affine normal coordinate
$$
v=z-x^{n-2}
$$
near $\Gamma$.
For the first $n$ moving sections one has $x^n=1$ at $t=0$ and the normal slope is
$$
v'=-(n-2)x^{n-2},
$$
while for the other $n$ sections one has $x^n=u^n$ and
$$
v'=(n-2)x^{n-2}.
$$
A section of $N(1)$ is represented on the affine chart by a pair of polynomials $(f,g)$ with
$$
\deg f\leq2n-3,\\;\; \deg g\leq1,
$$
satisfying the elementary-transform conditions at the $2n$ orbit points.
Equivalently, there exist polynomials $a,b$ of degree at most $n-3$ such that
\begin{align}
 f+(n-2)x^{n-2}g&=(x^n-1)a,\label{eq:normal1}\\
 f-(n-2)x^{n-2}g&=(x^n-u^n)b.\label{eq:normal2}
\end{align}
Subtracting \eqref{eq:normal2} from \eqref{eq:normal1}, the coefficients in degrees $n,\ldots,2n-3$ force $a=b$.
We then obtain
$$
2(n-2)x^{n-2}g=(u^n-1)a.
$$
The left-hand side is supported only in degrees $n-2$ and $n-1$, whereas the right-hand side has degree at most $n-3$.
Thus $g=a=b=0$, and then $f=0$.
Hence $H^0(N(1))=0$.

By Grothendieck splitting, write $N\simeq\cO(r)\oplus\cO(s)$.
The vanishing of $H^0(N(1))$ gives $r,s\leq-2$, while $r+s=-4$.
Therefore $r=s=-2$.
\end{proof}

We can now prove the multipoint theorem.

\begin{proof}[Proof of Theorem~\ref{thm:F0-main}]
On the special fiber, Proposition~\ref{prop:R-nef} gives the nef isotropic class $R_n$.
The curve $\Gamma_0$ has class $\gamma$ with $\gamma^2=-4$, and Lemma~\ref{lem:normal-bundle} gives
$$
N_{\Gamma_0/\mathcal X}\simeq\cO_{\PP^1}(-2)^{\oplus2}.
$$
Proposition~\ref{prop:balanced-reflection} therefore shows that $s_\gamma(R_n)$ is nef on very general fibers of the one-parameter family.
By Lemma~\ref{lem:reflected-ray}, a positive multiple of
$$
(n-4)H_1+H_2-\sqrt{\frac{n-4}{n}}\,E
$$
is nef on such fibers.

The same marked class is then nef on a very general free $G_n$-orbit.
Indeed, in the irreducible parameter space of free orbits, every integral class having negative intersection with the marked class has an effectivity locus which is closed and does not contain the orbit already obtained above.
There are only countably many integral classes.
The identical argument in the irreducible parameter space of ordered $2n$-tuples of distinct points shows nefness for $2n$ very general points of $S$.

Since
$$
\cO_S(n-4,1)^2=2(n-4),
$$
the volume bound is
$$
\sqrt{\frac{2(n-4)}{2n}}=\sqrt{\frac{n-4}{n}}.
$$
The constructed nef class gives the reverse inequality, so equality holds.
\end{proof}

\begin{remark}\label{rem:irrational-multipoint}
For $n>4$ the number $\sqrt{(n-4)/n}$ is irrational.
Indeed, if $n(n-4)=m^2$, then
$$
(n-2-m)(n-2+m)=4.
$$
For odd $n$ both factors are positive even integers, forcing both to equal $2$, which is impossible for $n>4$.
\end{remark}

\section{The dihedral quotient}\label{sec:quotient}

We now convert the orbit equality into a one-point equality.
It is convenient to conjugate the action \eqref{eq:group-original} by
$$
x'=x/u^{1/2},\\;\; z'=u^{(n-2)/2}/z.
$$
Dropping the primes, the generators become
\begin{equation}\label{eq:normalized-action}
r(x,z)=(\zeta x,\zeta^2z),\\;\; s(x,z)=(1/x,1/z).
\end{equation}
Let
$$
q:S\longrightarrow Y=S/G_n
$$
be the quotient.

The invariant functions
$$
x^n+x^{-n},\\;\; z^n+z^{-n}
$$
induce morphisms $\phi_1,\phi_2:Y\to\PP^1$.
Define the Cartier divisor
\begin{equation}\label{eq:A-on-Y}
A=(n-4)\phi_1^*\cO_{\PP^1}(1)+\phi_2^*\cO_{\PP^1}(1).
\end{equation}
Then
\begin{equation}\label{eq:A-pullback}
q^*A=\cO_S(2n(n-4),2n)=\cO_S(n-4,1)^{\otimes2n}.
\end{equation}
In particular $A$ is ample because its finite pullback is ample, and
$$
A^2=\frac{(q^*A)^2}{2n}=4n(n-4).
$$

\begin{corollary}\label{cor:quotient-epsilon}
For a very general smooth point $y\in Y$ one has
$$
\eps(A;y)=2\sqrt{n(n-4)}.
$$
\end{corollary}

\begin{proof}
Choose $p\in S$ with trivial stabilizer and put $y=q(p)$.
The blow-up of $S$ along the free orbit $G_n\cdot p$ maps finitely to $\Bl_yY$, and the exceptional divisor over $y$ pulls back to the sum of the $2n$ exceptional curves upstairs.
Nefness can be tested after finite pullback.
Using \eqref{eq:A-pullback} and Theorem~\ref{thm:F0-main}, we obtain
$$
\eps(A;y)=2n\sqrt{\frac{n-4}{n}}=2\sqrt{n(n-4)}.
$$
The assertion holds at a very general smooth point because the quotient of the very general free-orbit locus is very general in $Y$.
\end{proof}

We next analyze the singularities and the minimal resolution of $Y$.
Write
$$
n=2k+1.
$$
Every nontrivial rotation fixes the four corners of $S$.
They form two $G_n$-orbits, represented by $(0,0)$ and $(0,\infty)$.
The tangent weights are $(1,2)$ and $(1,n-2)$ modulo $n$, respectively.
Thus the corresponding quotient singularities are
$$
\frac1n(1,2),\\;\; \frac1n(1,n-2).
$$
Every reflection $r^js$ has four torus fixed points, and these $4n$ points form four $G_n$-orbits.
Their stabilizers have order two and act as $(-1,-1)$ on the tangent space, so their images are four $A_1$ singularities.
There are no further points with nontrivial stabilizer.
Hence $Y$ has precisely six singular points.

Let
$$
\nu:Z\longrightarrow Y
$$
be the minimal resolution.
The morphism $\phi_1$ induces a ruling
$$
f:Z\longrightarrow\PP^1.
$$
Its reducible fibers lie over $2,-2,\infty$.
Over each of the finite values $\pm2$, the fiber is the chain
$$
U_j\ --\ V_j\ --\ W_j,\\;\; j=1,2,
$$
with self-intersections and multiplicities
$$
(-2,-1,-2),\\;\; (1,2,1).
$$

At infinity, the Hirzebruch--Jung continued fractions are
\begin{equation}\label{eq:HJ}
\frac n2=[k+1,2],\\;\; \frac nk=[3,\underbrace{2,\ldots,2}_{k-1}],
\end{equation}
where the second expression is used because $k$ is the inverse of $n-2$ modulo $n$.
Consequently the fiber over infinity is the chain
\begin{equation}\label{eq:infinity-chain}
G_1\ --\ G_2\ --\ T\ --\ H_1\ --\ H_2\ --\cdots--\ H_k
\end{equation}
with self-intersections
\begin{equation}\label{eq:infinity-self}
-(k+1),\ -2,\ -1,\ -3,\ \underbrace{-2,\ldots,-2}_{k-1}
\end{equation}
and multiplicities
\begin{equation}\label{eq:infinity-mult}
1,\ k+1,\ n,\ k,\ k-1,\ldots,1.
\end{equation}
The multiplicities follow from the equations saying that the total fiber has intersection zero with every component, and these equations also give $T^2=-1$.

There is a distinguished section of $f$.
Let
$$
B=\{z=x^2\}\subset S.
$$
The curve $B$ is $G_n$-invariant for the normalized action \eqref{eq:normalized-action}, and the quotient $B/G_n$ maps isomorphically to the base of $f$.
Let $C\subset Z$ be its strict transform.

\begin{lemma}\label{lem:Cminusone}
The curve $C$ is a section of $f$ with
$$
C^2=-1.
$$
It meets $U_1$ and $U_2$ in the two finite reducible fibers and meets $G_1$ in the fiber over infinity.
\end{lemma}

\begin{proof}
The curve $B$ has class $2H_1+H_2$ and hence $B^2=4$.
Since $B\to B/G_n$ has degree $2n$, the image $\overline B\subset Y$ satisfies
$$
\overline B^2=\frac{4}{2n}=\frac2n.
$$
The curve $\overline B$ passes through one $A_1$ point in each finite reducible fiber, contributing a correction $1/2$ at each point to the self-intersection of its strict transform.
It also passes through the cyclic quotient point of type $\frac1n(1,2)$.
In the resolution chain $[-(k+1),-2]$, the curve $C$ meets the component $G_1$ transversely.
The corresponding diagonal entry of the inverse of the positive intersection matrix
$$
\begin{pmatrix}k+1&-1\\-1&2\end{pmatrix}
$$
is $2/n$.
Therefore
$$
C^2=\frac2n-\frac12-\frac12-\frac2n=-1.
$$
The assertions about the components met by $C$ follow from the local equations of $B$ and from the construction of the quotient fibers.
\end{proof}

\section{The plane model and deformation to very general points}\label{sec:plane}

We now contract fiber components disjoint from the section $C$.
In each finite reducible fiber contract first $V_j$ and then the image of $W_j$.
In the fiber over infinity contract successively
$$
T,\ G_2,\ H_1,\ H_2,\ldots,H_k.
$$
Each curve has self-intersection $-1$ at the moment it is contracted.
The resulting ruled surface contains the image of $C$ as a section of square $-1$, hence is $\F_1$.
Contracting this section gives a birational morphism
$$
Z\longrightarrow\PP^2
$$
whose inverse is a sequence of
$$
1+4+(k+2)=k+7
$$
point blow-ups.
Thus
$$
Z\simeq \Bl_{k+7}\PP^2
$$
with generally infinitely near centers.

Let
$$
h,e_1,\ldots,e_{k+7}
$$
be the resulting plane marking by total transform classes.
The first blow-up corresponds to the section, so
$$
C=e_1,\\;\; F=h-e_1,
$$
where $F$ is the class of a general fiber of $f$.
For the two finite reducible fibers we may choose the marking so that
\begin{equation}\label{eq:finite-plane-classes}
U_j=h-e_1-e_{2j}-e_{2j+1},\\;\; V_j=e_{2j+1},\\;\; W_j=e_{2j}-e_{2j+1}\quad (j=1,2).
\end{equation}
For the fiber at infinity, the classes are
\begin{align}
G_1&=h-e_1-e_6-e_7-\cdots-e_{k+6},\label{eq:G1class}\\
G_2&=e_{k+6}-e_{k+7},\\;\; T=e_{k+7},\label{eq:G2Tclass}\\
H_1&=e_{k+5}-e_{k+6}-e_{k+7},\label{eq:H1class}\\
H_j&=e_{k+6-j}-e_{k+7-j}\\;\; (2\leq j\leq k).\label{eq:Hjclass}
\end{align}
One checks directly that the weighted sums from \eqref{eq:infinity-mult} recover $F$.

Put
$$
M=\nu^*A\in\Pic(Z).
$$
The class $M$ is orthogonal to every exceptional curve of $\nu$.
The curves $V_j$ and $T$ are not exceptional for $\nu$, and no orthogonality is imposed on them.
The projection formula gives
\begin{equation}\label{eq:MC-MF}
M\cdot C=n-2,\\;\; M\cdot F=2n.
\end{equation}
Indeed, $q^*A=\cO_S(2n(n-4),2n)$, the invariant curve $B$ has class $(2,1)$ and maps to its quotient with degree $2n$, while the pullback of a general fiber of $f$ is the sum of $2n$ first-ruling fibers.

The equations \eqref{eq:finite-plane-classes}--\eqref{eq:MC-MF} determine $M$ uniquely.

\begin{proposition}\label{prop:M-plane}
In the plane marking above,
\begin{equation}\label{eq:M-plane}
M=(3n-2)h-(n-2)e_1-n(e_2+e_3+e_4+e_5)-4(e_6+\cdots+e_{k+5})-2(e_{k+6}+e_{k+7}).
\end{equation}
In particular,
$$
M^2=4n(n-4)=A^2.
$$
\end{proposition}

\begin{proof}
Write
$$
M=dh-\sum_{i=1}^{k+7}m_ie_i.
$$
Since $C=e_1$ and $F=h-e_1$, equations \eqref{eq:MC-MF} give
$$
m_1=n-2,\\;\; d-m_1=2n,
$$
so $d=3n-2$.
Orthogonality to $W_1,W_2$ gives
$$
m_2=m_3,\\;\; m_4=m_5,
$$
and orthogonality to $U_1,U_2$ then gives
$$
m_2=m_3=m_4=m_5=n.
$$

At infinity, orthogonality to $G_2$ gives $m_{k+6}=m_{k+7}$, orthogonality to $H_1$ gives $m_{k+5}=2m_{k+6}$, and the remaining $H_j$ show that
$$
m_6=\cdots=m_{k+5}=2m_{k+6}.
$$
Finally, $M\cdot G_1=0$ yields
$$
\sum_{i=6}^{k+6}m_i=2n.
$$
Since $n=2k+1$, this becomes
$$
(2k+1)m_{k+6}=2n,
$$
so $m_{k+6}=m_{k+7}=2$ and $m_6=\cdots=m_{k+5}=4$.
This proves \eqref{eq:M-plane}.
The self-intersection is either computed directly or follows from $M=\nu^*A$.
\end{proof}

Choose a very general smooth point $y\in Y$ satisfying Corollary~\ref{cor:quotient-epsilon}, and let $a=\nu^{-1}(y)$.
Blow up $a$ and denote the new exceptional class by $e_{k+8}$.
The class
\begin{equation}\label{eq:Dspecial}
D=M-2\sqrt{n(n-4)}\,e_{k+8}
\end{equation}
is nef, because it is the pullback to the resolution of the nef Seshadri class on $\Bl_yY$.

Consider the irreducible parameter space of ordered sequences of $k+8$ point blow-ups of $\PP^2$, allowing infinitely near centers.
It carries the natural marking $h,e_1,\ldots,e_{k+8}$, and the special surface just constructed is one of its fibers.
The locus where all centers are distinct proper points of $\PP^2$ is dense and open.
As in the proof of Proposition~\ref{prop:balanced-reflection}, the effectivity locus of any fixed integral marked class having negative intersection with \eqref{eq:Dspecial} is a proper closed subset because it does not contain the special fiber.
There are only countably many integral classes.
Hence the marked class \eqref{eq:Dspecial} is nef on the blow-up of $k+8$ very general distinct plane points.

We now make one quadratic transformation, centered at three of the four points carrying multiplicity $n$ in \eqref{eq:M-plane}, and reorder the first $k+7$ points afterwards.
The coefficient of the last point $e_{k+8}$ is unchanged.
The first $k+7$ coefficients transform as
$$
(3n-2;n^4,n-2,4^k,2^2)\longmapsto(3n-4;n,(n-2)^4,4^k,2^2),
$$
where superscripts denote repetitions.
Thus, on the blow-up of $k+8$ very general plane points, the class
\begin{equation}\label{eq:nef-Ln}
(3n-4)H-nE_1-(n-2)(E_2+\cdots+E_5)-4(E_6+\cdots+E_{k+5})-2(E_{k+6}+E_{k+7})-2\sqrt{n(n-4)}E_{k+8}
\end{equation}
is nef.

Fixing the first $k+7$ very general points and interpreting the last one as a very general point of their blow-up, \eqref{eq:nef-Ln} gives
\begin{equation}\label{eq:epsilon-lower}
\eps(L_n;x)\geq2\sqrt{n(n-4)}.
\end{equation}
A direct calculation gives
\begin{equation}\label{eq:Lsquare}
L_n^2=(3n-4)^2-n^2-4(n-2)^2-16k-8=4n(n-4).
\end{equation}
Thus the right-hand side of \eqref{eq:epsilon-lower} is the volume bound.
It remains only to verify that $L_n$ is ample.

\section{Ampleness and proof of the main theorem}\label{sec:ampleness}

We use the ampleness criterion of Hanumanthu \cite[Theorem~2.1]{Hanumanthu}.
For convenience, we recall only the part that we need.

\begin{theorem}[Hanumanthu]\label{thm:Hanumanthu}
Let $X\to\PP^2$ be the blow-up at $r$ general points and let
$$
L=dH-\sum_{i=1}^r m_iE_i,
$$
where $m_1\geq\cdots\geq m_r>0$.
Then $L$ is ample if the following inequalities hold:
\begin{align*}
d&>m_1+m_2,\\
2d&>m_1+\cdots+m_5,\\
3d&>2m_1+m_2+\cdots+m_7,
\end{align*}
and, for every $2\leq s\leq r$,
$$
d^2\geq\frac{s+3}{s+2}\sum_{i=1}^s m_i^2.
$$
\end{theorem}

\begin{proposition}\label{prop:ample}
For every odd $n=2k+1\geq7$, the divisor $L_n$ from \eqref{eq:Ln} is ample on the blow-up of $k+7$ general points of $\PP^2$.
\end{proposition}

\begin{proof}
Order the multiplicities as
$$
n,\quad (n-2)^4,\quad 4^k,\quad 2^2.
$$
Set $d=3n-4$.
The first three inequalities in Theorem~\ref{thm:Hanumanthu} become
\begin{align*}
3n-4&>2n-2,\\
6n-8&>5n-8,\\
9n-12&>6n,
\end{align*}
and all hold for $n\geq7$.

For the last family of inequalities, observe that for every $s\geq2$,
$$
\frac{s+3}{s+2}\leq\frac54.
$$
Moreover the sum of the squares of all multiplicities is
$$
\Sigma=n^2+4(n-2)^2+16k+8=5n^2-8n+16.
$$
Hence it is enough to prove
$$
d^2>\frac54\Sigma.
$$
After clearing denominators this is
$$
4(3n-4)^2-5(5n^2-8n+16)=11n^2-56n-16>0.
$$
The last expression is positive at $n=7$ and is strictly increasing for $n\geq7$.
Thus all hypotheses of Theorem~\ref{thm:Hanumanthu} are satisfied.
\end{proof}

\begin{proof}[Proof of Theorem~\ref{thm:plane-main}]
Proposition~\ref{prop:ample} proves that $L_n$ is ample.
The nef class \eqref{eq:nef-Ln} gives
$$
\eps(L_n;x)\geq2\sqrt{n(n-4)}
$$
at a very general point.
By \eqref{eq:Lsquare}, the volume bound gives the reverse inequality.
Therefore
$$
\eps(L_n;x)=2\sqrt{n(n-4)}.
$$
Finally, $n(n-4)$ is not a square for $n>4$ by Remark~\ref{rem:irrational-multipoint}, so the Seshadri constant is irrational.
\end{proof}

\begin{remark}[The natural plane-model divisor]\label{rem:natural-divisor}
Before the quadratic transformation one obtains the equally explicit divisor
$$
M_n^{\mathrm{pl}}=(3n-2)H-n(E_1+E_2+E_3+E_4)-(n-2)E_5-4(E_6+\cdots+E_{k+5})-2(E_{k+6}+E_{k+7}),
$$
after reordering the coefficients of \eqref{eq:M-plane}.
It also has square $4n(n-4)$ and the same maximal Seshadri constant at a very general point.
The quadratic transformation is used only to obtain the more reduced expression \eqref{eq:Ln}.
For $n=5$, two quadratic transformations starting from the natural plane-model class recover the divisor $9H-3(E_1+\cdots+E_5)-2(E_6+\cdots+E_9)$ of Laface--Ugaglia.
\end{remark}

\section{Further remarks and possible extensions}

The argument separates naturally into four ingredients:
\begin{enumerate}[label=\textup{(\roman*)}]
\item a special configuration carrying an explicitly constructed nef isotropic ray;
\item a rational curve $\Gamma$ of negative even self-intersection;
\item a deformation in which $N_{\Gamma/\mathcal X}$ is balanced in the precise sense required by Proposition~\ref{prop:balanced-reflection};
\item an identification of the reflected isotropic ray with a Seshadri ray.
\end{enumerate}
Only the first and fourth ingredients depend strongly on the chosen surface.
This suggests that the balanced-reflection principle may be useful in settings quite different from $\PP^1\times\PP^1$.

In the present construction the numerical coincidence
$$
\Gamma^2=-4
$$
is independent of $n$ because the curve $\Gamma$ has class $(n-2,1)$ and contains $2n$ orbit points:
$$
2(n-2)-2n=-4.
$$
At the same time, the normal bundle in the threefold total space remains
$$
\cO_{\PP^1}(-2)^{\oplus2}
$$
for every odd $n$.
Thus the reflection mechanism is stable in the parameter $n$. The changing arithmetic is entirely absorbed by the de Jonqui\`eres dynamics and by the quotient singularity $\frac1n(1,n-2)$.

It would be interesting to search for analogous constructions on other ruled surfaces, on rational surfaces carrying Coxeter-type automorphisms, and on families in which a rational curve of square $-2a$ has balanced normal bundle $\cO(-a)^{\oplus2}$.
The abstract formulation in Proposition~\ref{prop:balanced-reflection} indicates that a finite group quotient is not essential to the reflection step itself.
Its role here is to convert a multipoint equality upstairs into a one-point equality downstairs and to provide a concrete rational plane model after resolution.

\begin{question}
Can the balanced-reflection principle be combined with a dynamical nef ray on a smooth family in such a way that the quotient step is avoided altogether?
In particular, can one obtain infinite families of irrational one-point Seshadri constants on smooth surfaces of fixed Picard number?
\end{question}

\begin{question}
For even values of $n$, can one replace the present choice of weights by a different dihedral or more general finite-group action so that the de Jonqui\`eres iteration remains nondegenerate and produces an analogous irrational limiting ray?
\end{question}

We hope to come back to these questions in a future work.

\paragraph{Acknowledgements.} We are deeply grateful to Piotr Pokora and Tomasz Szemberg for many inspiring conversations that motivated parts of this research, as well as for their valuable comments.

The present work was carried out during a Workshop on Algebraic Geometry held in September 2026. We thank Dwór Sieraków for hospitality and the University of the National Education Commission for financial support under grant DNa.711.404.2026.PN.

\noindent
Grzegorz Malara, Łukasz Merta, Justyna Szpond and Marcin Zieli\'nski\\
Department of Mathematics,
University of the National Education Commission, Krakow,
Podchor\c a\.zych 2,
PL-30-084 Krak\'ow, Poland.\\
\textit{E-mail address:} \texttt{grzegorzmalara@gmail.com}\\
\textit{E-mail address:} \texttt{lukasz.merta@uken.krakow.pl}\\
\textit{E-mail address:} \texttt{szpond@gmail.com}\\
\textit{E-mail address:} \texttt{marcin.zielinski@uken.krakow.pl}

\end{document}